\documentclass[11pt]{amsart}

\usepackage{amsmath,a4wide,exscale,amsthm} 
\usepackage{amssymb,latexsym}
\usepackage{enumerate}

\usepackage[pdftex]{graphicx}

\theoremstyle{definition}

\theoremstyle{plain}
\newtheorem{proposition}{Proposition}
\newtheorem{corollary}{Corollary}
\newtheorem{theorem}{Theorem}
\newtheorem{lemma}{Lemma}

\newcommand{\Z}{\mathbb{Z}}
\newcommand{\R}{\mathbb{R}}

\newcommand{\N}{\mathbb{N}}

\newcommand{\mc}[1]{\mathcal{#1}}
\DeclareMathOperator{\Si}{Si}

\title{Concentration inequalities for Paley-Wiener spaces}

\author{Syed Husain, Friedrich Littmann}

\begin{document}

\maketitle

\begin{abstract} This article considers the question of how much of the mass of an element in a Paley-Wiener space can be concentracted on a given set. We seek bounds in terms of  relative densities of the given set.  We extend a result of Donoho and Logan from 1992 in one dimension and consider similar results in higher dimensions.
\end{abstract}

\section{Introduction}


 Let $M$ be a convex body in $\R^d$, and let $\mc{B}_p(M)$, $1\le p\le \infty$, be the Paley-Wiener space of elements from $L^p(\R^d)$ with distributional Fourier transform supported in $M$. The Fourier transform $\mc{F} f$ is given by
\[
\mc{F}f(\varphi)  = \int_\R \widehat{\varphi}(t) f(t) dt
\]
 for a Schwarz function $\varphi$.  (We use $\widehat{\varphi}(t) = \int \varphi(x) e^{-2\pi i x t} dx$.) We write $\mc{B}_p(\tau)$ if $M$ is the ball with center at the origin and radius $\tau$. 

Let $N$ and $W_\delta\subseteq \R^d$ be measurable and set $W_\delta(x) = x+ W_\delta$. (In this article, $W_\delta$ is either a ball or a cube.)  We consider the problem of finding a constant $C(M,\delta)>0$ such that
\begin{align}\label{intro-mass}
\|G \chi_N\|_1 \le C(M,\delta) \sup_{x\in\R^d} |N\cap W_\delta(x)| \, \|G\|_1\text{ for all }G\in \mc{B}_1(M).
\end{align}

Here $|.|$ denotes Lebesgue measure and $\chi_N$ is the characteristic function of $N$. We emphasize that the constant is  not allowed to depend on  $N$.

This question was studied by Donoho and Logan \cite{DL} in dimension $d=1$ in connection with recovery of a bandlimited signal that is corrupted by noise. In their setting, an unknown noise $n\in L^1(\R)$ is added to a known signal $F\in \mc{B}_1([-\tau,\tau])$, and they investigate sufficient conditions under which the best approximation $\widetilde{F}\in \mc{B}_1([-\tau,\tau])$ to $F+n$ satisfies $\widetilde{F} = F$, i.e., when $F$ can be perfectly recovered from knowledge of $F+n$ through  best $L^1$-approximation.

Denoting now by $N$ the support of $n$, it is a remarkable fact that the concentration condition
\begin{align}\label{logan-phenom}
\frac{\|G \chi_N\|_1 }{\| G\|_1 } <\frac12\textrm{ for all $G\in \mc{B}_1(M)$}
\end{align}
 is  sufficient to conclude that $F = \widetilde{F}$. The argument can be found in several places, e.g., Donoho and Stark \cite[Section 6.2]{DS}, who refer to it as Logan's phenomenon. (Logan's thesis \cite{L65} appears to contain the earliest record of this argument.) It was shown in \cite[Theorem 7]{DL} that \eqref{intro-mass} holds for $W_\delta(x) = [x-\frac\delta2,x+\frac\delta2]$ with
\begin{align}\label{DL-result}
C([-\tau,\tau],\delta) = \frac{\pi\tau}{\sin(\pi\tau\delta)},
\end{align}
and combining this with \eqref{logan-phenom}, it is evident that this gives $F = \widetilde{F}$ provided the relative density (or Nyquist density) of the support of the noise satisfies
\[
\delta^{-1}\sup_{x\in \R} |N\cap[x,x+\delta]| <\frac{\sin(\pi\tau\delta)}{2\pi\tau \delta}.
\]

We mention that conditions to recover an element of a closed subspace of an $L^1$ space that has been corrupted by a sparse $L^1$-noise have been investigated in many different settings, and concentration inequalities lead frequently to sufficient conditions. (This relies on the fact that if a set $N$ satisfies an analogue of \eqref{logan-phenom} for all $G$ in a given closed subspace of an $L^1$-space, then the zero function is the closest element from the subspace to every $L^1$ function with support contained in $N$.) For interested readers we refer to  Cand\`es, Romberg, and Tao \cite{CRT06}, Benyamini, Kro\'o, and Pinkus  \cite{BKP12}, Abreu and Speckbacher \cite{AS21}, and the references therein.

\section{Results}

There are two questions that this article seeks to address. First, it is clear that the shape of the  bound in \eqref{DL-result} requires $\delta\tau<1$. In contrast, it was shown for $p=2$ in \cite[Theorem 4]{DL} that for any positive $\tau$ and $\delta$
\[
\|G\chi_N\|_2^2 \le \left(\tau+\delta^{-1}\right) \sup_{x\in\R} \left|N\cap[x,x+\delta]\right| \, \|G\|_2^2
\]
for all $G\in \mc{B}_2(\tau)$ (with constants adjusted due to the different normalization of the Fourier transform) which suggests that an inequality with constant $c(\tau+\delta^{-1})$ should also be true for $p=1$. Our first result confirms that this is the case.

\begin{theorem}\label{thm1} Let $N$ be the support of $n\in L^1(\R)$. Then for all $G\in \mc{B}_1(\tau)$
\[
\|G\chi_N\|_1 \le C_{\tau,\delta} \sup_{x\in \R} \left|N\cap [x,x+\delta]\right|\, \|G\|_1.
\]
where $C_{\tau,\delta}\le \frac{80}{13}(\tau + \delta^{-1})$ for all positive $\tau$ and $\delta$. The bound may be improved to $C_{\tau,\delta}\le \frac52(\tau+\delta^{-1})$ for $\tau\delta\ge 2$. Moreover, $\lim_{\tau\delta\to\infty} C_{\tau,\delta} =2$. 
\end{theorem}

As is usual with this method, the bounds only become effective when the density is a fraction of the reciprocal of the type $\tau$. If one is interested in bounds for $\|G\chi_N\|_1/ \|G\|_1$ at larger densities, a version of the Logvinenko-Sereda theorem from O. Kovrijkine \cite{Kov01} gives non-trivial bounds whenever the density is smaller than $1$\footnote{The authors are grateful to Walton Green to draw their attention to \cite{Kov01}.}. The constants are not effective and don't yield concrete bounds to decide when the quotient is $<1/2$. 

Our second result deals with reconstruction in higher dimensions. We investigate the case when $M$ is a cube and $W_\delta(0)$ is a ball with center at the origin, and we indicate the obstructions that we encountered when taking $M$ to be a ball with center at the origin. We denote by $J_\nu$ the Bessel function of the first kind and by $j_\nu(k)$ its $k$th positive zero.

\begin{theorem}\label{thm2}
Let $d\in\N$ and  let $N\subseteq \R^d)$. If $\alpha\lambda<j_{d/2}(1) \, d^{-\frac12},$ then for all $G\in \mc{B}_1([\lambda,\lambda]^d)$
\begin{align*}
 \|G\chi_N\|_1
&\leq\frac{(\sqrt{d}\lambda)^{d/2}}{\alpha^{d/2}J_{d/2}(2\pi\sqrt{d}\alpha\lambda)}\sup_{x\in\R^d}\left| N\cap B(x,\alpha)\right| \, \|G\|_1.
\end{align*}
\end{theorem}

Analogously to dimension one, for a convex body $K$ we define the maximum Nyquist density of $N$ (relatively to $K$) by 
\[
\rho(N,K)=\frac{1}{|K|}\sup\limits_{u\in \R^d}|N\cap (u+K)|.
\]


We compare the result of Theorem \ref{thm2} to the case where the window $K$ is a hypercube of side length $\delta$, which is an extension of the $L_1$ reconstruction result by \cite{DL}.  The zero $j_{p}(1)$ has an asymptotic expansion given in \cite{NIST:DLMF} by $$j_{p}(1)\simeq \left(\frac{p}{2}+\frac{1}{4}\right)\pi.$$
Denote the ball of radius $r$ centered at origin by $B(0,r)$, and the volume of a ball with radius $\alpha$ in $d$-dimensions by $V_d(\alpha)$. It is given by $V_d(\alpha)=\frac{\pi^{d/2}}{\Gamma(d/2+1)}\alpha^d.$
When the window is a ball of radius $\alpha$, perfect reconstruction is possible if the maximum Nyquist density satisfies
\[
\rho(N,B(0,\alpha))<\frac{\Gamma(d/2+1)J_{d/2}(2\pi\sqrt{d}\alpha\lambda)}{2(\pi\alpha\sqrt{d}\lambda)^{d/2}},
\]
 where $\alpha<\frac{j_{d/2}(1)}{2\pi\sqrt{d}\lambda}$.
For $\lambda\delta<2\pi$, the corresponding density bound is
\[
\rho(N,[-\delta/2,\delta/2]^d)<\frac{1}{2}\left(\frac{\sin(\lambda\delta/2)}{\lambda\delta/2}\right)^d.
\]

The support of the Fourier transform for both the problems is same, that is, $[-\lambda,\lambda]^d$.
In the second case, we consider the ball just outside the cube such that the radius satisfies $\alpha=\frac{\delta\sqrt{d}}{2}$. Let $\delta=\frac{1}{2\pi^2}.$ For large $d$, the bound is asymptotically
\begin{align*}
\rho(T,B(0,\alpha),F)&<\frac{\Gamma(d/2+1)J_{d/2}(d/2)}{2(d/4)^{d/2}}\\
&\sim \frac{\sqrt{\pi d}\left(\frac{d}{2e}\right)^{d/2}}{2\left(\frac{d}{4}\right)^{d/2}}\frac{\Gamma(1/3)}{2^{1/3}\cdot 3^{1/6}\cdot \pi.d^{1/3}}\\
&\sim d^{1/6}\left(\frac{2}{e}\right)^{d/2}
\end{align*}
The Nyquist density for the cube window satisfies
$$\rho(T,[-\delta/2,\delta/2]^d,F)<\frac{1}{2}\left(4\pi^2\sin(1/4\pi^2)\right)^d.$$
 The bound for the Nyquist density of the cube window remains larger than the bound for the Nyquist density of ball window for any $d$ in this case.

Third, we set the volume of the cube is equal to the volume of the ball. Then the radius $\alpha$ of the ball satisfies
$$\alpha=\delta\sqrt[d]{\frac{\Gamma(d/2+1)}{\pi^{d/2}}}.$$
Using Sterling's approximation, we get
$$\alpha\simeq \delta\sqrt[d]{\left(\frac{d}{2\pi e}\right)^{d/2}(\pi d)^{1/2}}.$$
Let $\delta=\frac{\sqrt{2\pi e}}{4\pi^2}$. For large $d$, the Bessel function in the Nyquist density of the ball window satisfies
$$J_{d/2}(2\pi^2\sqrt\alpha)=J_{d/2}(d\cdot\pi^{1/2d}\cdot d^{1/2d}/2)\rightarrow J_{d/2}(d/2),$$
since $\pi^{1/2d}\cdot d^{1/2d}\rightarrow 1$ for large $d$. The bound for the  Nyquist density of the ball window is then
\begin{align*}
\rho(T,B(0,\alpha),F)&<\frac{\Gamma(d/2+1)J_{d/2}(d/2)}{2\left(\pi^2\sqrt{d}\frac{\sqrt{2\pi e}}{4\pi^2 \sqrt{\pi}}\sqrt[d]{\Gamma(d/2+1)}\right)}\\
&\sim\frac{\Gamma(1/3).\pi^{1/4}}{2\cdot 2^{1/3}\cdot 3^{1/6}\cdot \pi}\left(\frac{4}{2e}\right)^{d/2}\frac{1}{d^{1/12}}
\end{align*}

Tor the cube window, the sufficient bound for reconstruction is
$$\rho(T,[-\delta/2,\delta/2]^d,F)<\frac{1}{2}\left(\frac{\sin(\sqrt{2\pi e}/8\pi)}{\sqrt{2\pi e}/8\pi}\right)^d.$$
In this case also, the Nyquist density for the cube window remains larger than the Nyquist density of ball window for any dimension $d$.

\section{Proof of Theorem \ref{thm1}} We briefly review a general approach to prove inequalities of the above form introduced by Donoho and Logan in \cite{DL}. Construct a kernel $K(x,y)$ so that $f\mapsto Tf$ given by
\[
Tf(y) = \int K(x,y) f(x) dx
\]
defines a bounded invertible transformation when restricted to $\mc{B}_1(\tau)$. Then a change of integration order gives
\begin{align*}
\int_N |G(x)| dx &\le \int_N  \int|K(x,y)| T^{-1}G(x)| dx dy \\
&\le \left(\sup_x \int_N K(x,y) dy\right) \|T^{-1}\| \|G\|_1.
\end{align*}

If $K(x,y) = g(x-y)$ for some $g\in L^\infty$ with ${\rm supp}(g)\subseteq W_\delta(0)$, then the supremum may be further estimated by $\|g\|_\infty \sup_x |N\cap W_\delta(x)|$, where $T=T_g$ is now  the convolution operator $T_gf = f*g$ restricted to $\mc{B}_1(\tau)$. For given $g$ the size of the constant depends then only on $\|g\|_\infty \|T_g^{-1}\|$, and \eqref{logan-phenom} shows that we need
\[
\sup_x |N\cap W_\delta(x)| <\frac1{2 \|g\|_\infty \|T_g^{-1}\|}. 
\]

Thus, it is the task to construct $g$ as above where $\|g\|_\infty \|T_g^{-1}\|$ is as small as possible. The choice in \cite{DL} was $g = \chi_{[-\frac\delta2,\frac\delta2]}$, which is optimal for $\delta\tau\le\frac12$, gives a non-optimal bound for $\frac12<\delta\tau<1$, and fails to give a bound for $\delta\tau\ge 1$. This can be traced back to the fact that $\widehat{g}(\delta) =0$.

\medskip

To create an auxiliary function $g$ with computable product $\|g\|_\infty \|T_g^{-1}\|$, Logan and Donoho observed that if $1/\widehat{g}$ is positive and convex up on an interval $I = [-a,a]$ with center at the origin, then the periodic extension of $1/\widehat{g}$ restricted to $I$ is the Fourier transform of a measure $\nu$ that acts as the inverse operator of convolution with $g$ on $\mc{B}_1(a)$ and has total variation $|\nu| = 1/\widehat{g}(a)$. (In fact, $\nu$ is the minimal extrapolation of $1/\widehat{g}$ restricted to $I$ in the sense of Beurling.) Our choice of $g$ is based on this idea. We define for $\tau>0$ and real $x$ a function $g_\tau$, supported on $[-1,1]$, by
\[
g_\tau(x) = -2 (1-|x|) \frac{\cos2\pi (\tau+1) x - \cos2\pi \tau x}{4\pi^2 x^2} \chi_{[-1,1]}(x).
\]

The Fourier transform of $g_\tau$ has the useful property that  the sum of its  partials with respect to $t$ and $\tau$ has a simple integral representation.

\begin{figure}[t]
\includegraphics[width=6.5cm]{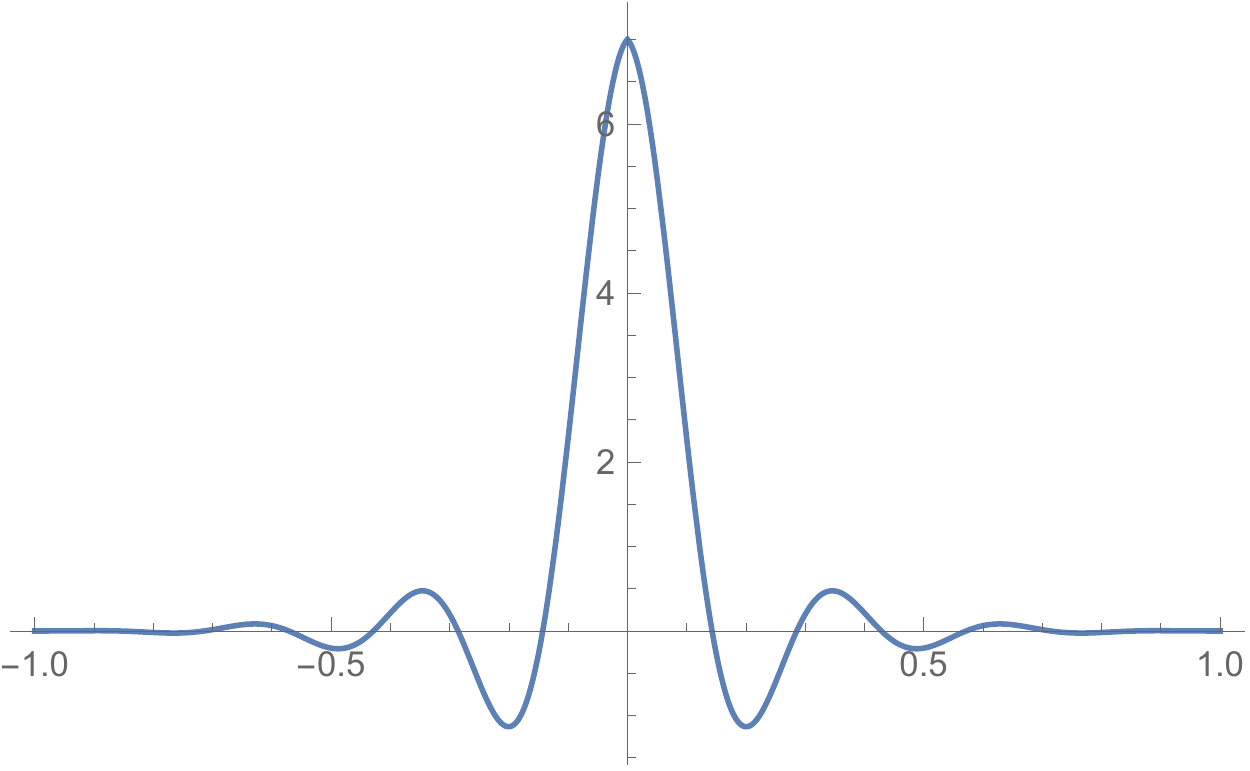}
\includegraphics[width=6.5cm]{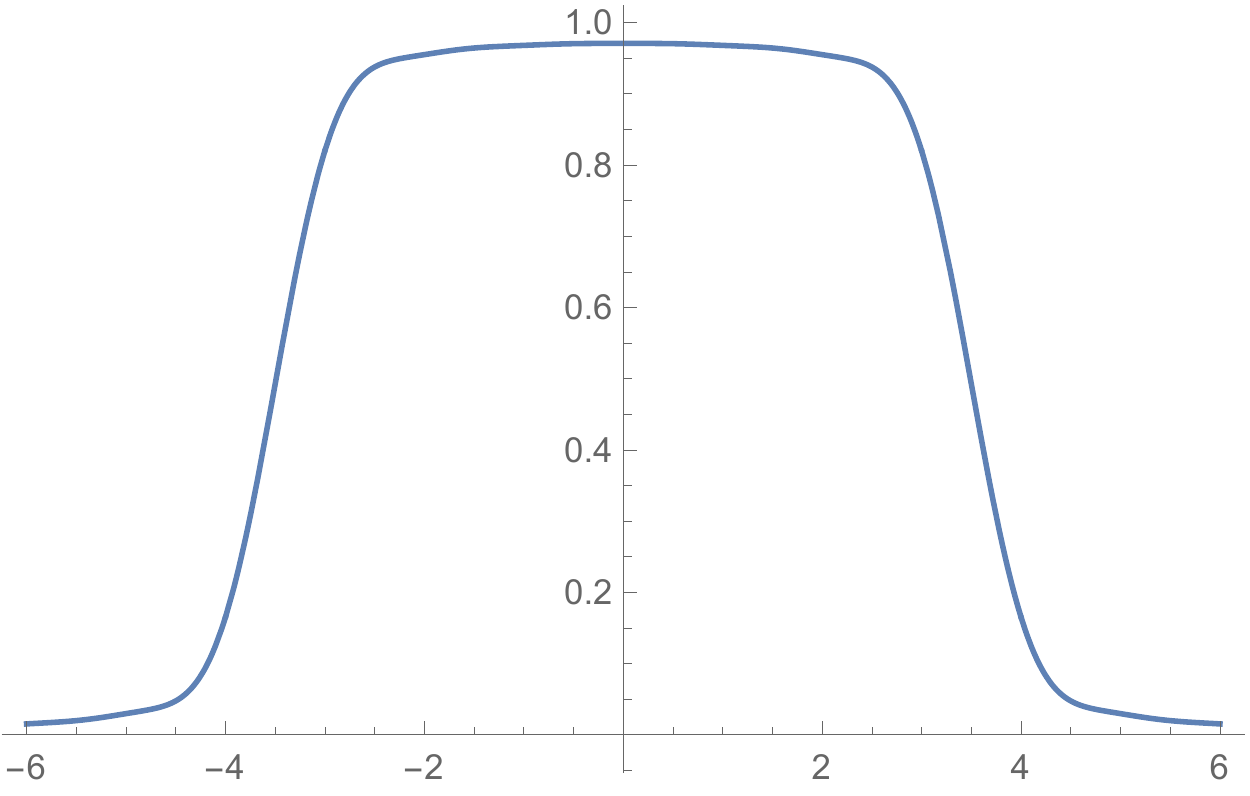}
\caption{The transform pair $g_4(x)$ and $\widehat{g}_4(t)$}
\label{fig2}
\end{figure}

\begin{proposition}\label{first-partials}  For any $t$ and $\tau$
\[
\frac{\partial}{\partial t} \left(\widehat{g}_\tau(t) \right) + \frac{\partial}{\partial \tau} \left(\widehat{g}_\tau(t)\right) = \int_{2\pi(t+\tau)}^{2\pi(t+\tau+1)} \frac{\sin^2 u}{u^2} du.
\]
\end{proposition}

\begin{proof}  For ease of notation we set $G(t,\tau) = \widehat{g}_\tau(t)$, and we denote first  partials by $G_t$ and $G_\tau$. Writing 
\[
g_\tau(x) = 2(1-|x|) \frac{\cos 2\pi (\tau+1)x - \cos2\pi \tau x}{(-2\pi ix)^2}\chi_{[-1,1]}(x)
\]
and using that $g_\tau(x)$ is even, we have
\begin{align*}
G_t(t,\tau) &= \int_{-1}^1  (-2\pi i x)  g_\tau(x) e^{-2\pi i x t}dx \\
&= \int_{-1}^1 (-2\pi i x) g_\tau(x) (-i\sin(2\pi x t)) dx \\
&= \int_{-1}^1 2(1-|x|) \frac{\cos(2\pi(\tau+1) x) - \cos(2\pi \tau x)}{2\pi x} \sin(2\pi x t) dx\\
&= 4\int_0^1 (1-x) \frac{\cos(2\pi(\tau+1) x) - \cos(2\pi \tau x)}{2\pi x} \sin(2\pi x t) dx.
\end{align*}
 
Similarly,
\begin{align*}
G_\tau(t,\tau) &= \int_{-1}^1 \frac{\partial}{\partial \tau} (g_\tau(x)) e^{-2\pi i x t} dx \\
&= 4 \int_{0}^1 (1-x) \frac{\sin(2\pi(\tau+1) x) - \sin(2\pi \tau x)}{2\pi x} \cos(2\pi x t) dx.
\end{align*}

The integrals have representations in terms of the sine-integral $\Si(u) = \int_0^u \sin(w)/w dw$. A direct calcuation gives
\begin{align*}
2\int_0^1 \cos(2\pi a x) \frac{\sin(2\pi b x)}{x} dx &= \Si(2\pi(a+b)) - \Si(2\pi(a-b))\\
2\int_0^1 \cos(2\pi a x) \sin(2\pi b x)dx&= -\frac{b}{\pi(a-b) (a+b)} +\frac{\cos(2\pi(a-b))}{2\pi(a-b)} -\frac{\cos(2\pi (a+b))}{2\pi(a+b)}.
\end{align*} 

We obtain
\begin{align*}
G_t(t,\tau) + G_\tau(t,\tau) &= \frac2{\pi^2} \left(  \frac{\sin^2 (\pi(t+\tau))}{(t+\tau)(t+\tau+1)}  +\pi \Si(2\pi(t+\tau+1))-\pi \Si( 2\pi(t+\tau) \right)\\
&= \frac2\pi \int_{\pi(t+ \tau)}^{\pi (t+\tau+1)} \left(-\frac{\partial}{\partial u} \frac{\sin^2 u}{u}\right) du+ \frac2\pi \int_{2\pi (t+\tau)}^{2\pi (t+\tau+1)} \frac{\sin w}{w} dw\\
&=\frac2\pi \int_{\pi(t+\tau)}^{\pi(t+\tau+1)} \frac{\sin^2 u}{u^2} du
\end{align*}
after substituting $w = 2u$ and combining the integrands.  
\end{proof}

\begin{corollary}
\begin{enumerate}
\item\label{pt-two} $\|g_\tau\|_\infty = g_\tau(0) =2\tau+1$.
\item\label{pt-three} The function $\tau\mapsto \widehat{g}_\tau(0)$ is positive, monotonically increasing, and has limit $1$ as $\tau\to \infty$. Moreover,
\begin{align*}
\widehat{g}_0(0)>0.65, \qquad \widehat{g}_1(1) >0.8.
\end{align*}
\item\label{pt-four} $t\mapsto \big(\widehat{g}_\tau(t)\big)^{-1}$ is positive and convex (up) on $[-\tau,\tau]$.
\end{enumerate}
\end{corollary}

\begin{proof}   Property (\ref{pt-two}) is obtained by direct calculation.  For the proof of  (\ref{pt-three}) it follows from symmetry of $t\mapsto \widehat{g}_\tau(t)$ that  $\widehat{g}_\tau'(0) =0$, and hence Proposition \ref{first-partials} gives $\partial/\partial \tau \widehat{g}_\tau(0)>0$. Direct calculations give the claimed bounds.

 Regarding (\ref{pt-four}),  we require an explicit representation of $\widehat{g}_\tau''(t)$. It follows from 
\[
g_\tau(x) = 2(1-|x|) \frac{\cos 2\pi (\tau+1)x - \cos2\pi \tau x}{(-2\pi ix)^2}\chi_{[-1,1]}(x)
\]
that 
\begin{align}\label{gk-2nd}
\begin{split}
 \widehat{g}_\tau''(t) &= \int_{-1}^1 (-2\pi i x)^2 g_\tau(x) e^{-2\pi i x t} dx\\
&= -\left(\frac{\sin \pi (t-\tau)}{\pi (t-\tau)}\right)^2 -  \left(\frac{\sin \pi (t+\tau)}{\pi (t+\tau)}\right)^2 \\
&\qquad \qquad \qquad \qquad+ \left(\frac{\sin \pi (t-\tau-1)}{\pi (t-\tau-1)}\right)^2 +  \left(\frac{\sin \pi (t+\tau+1)}{\pi (t+\tau+1)}\right)^2 \\
&= \frac{\sin^2(\pi(t-\tau)) (2(t-\tau)-1)}{(t-\tau)^2 (t-\tau-1)^2} - \frac{\sin^2(\pi(t+\tau))(2(t+\tau)+1)}{(t+\tau)^2 (t+\tau+1)^2}.
\end{split}
\end{align}
 
 Since the first term is negative for $t-\tau <1/2$ and the second term is positive for $t+\tau>-1/2$, it follows  that
\[
\widehat{g}_\tau''(t)<0\text{ for } -\tau -\tfrac12< t < \tau+\frac12.
\]

Multivariate chain rule and Proposition \ref{first-partials} show that  
\[
\frac{\partial}{\partial\tau} \left( \widehat{g}_\tau(\tau) \right) >0,
\]
 and since $\widehat{g}_0(0)>0$, it follows that $\widehat{g}_\tau(\tau)>0$ for all $\tau$. Since $\widehat{g}_\tau$ is concave down on $[-\tau,\tau]$, it follows that $\widehat{g}_\tau(t)>0$  for $t\in [-\tau,\tau]$. It follows that the second derivative of $t\mapsto (\widehat{g}_\tau(t))^{-1}$ is positive for $|t|\le \tau$. 
\end{proof}




\begin{proof}[Proof of Theorem \ref{thm1}]  Setting $g_{\tau,\delta}(x) = g_{\tau\delta/2}(2x/\delta)$, we observe that $g_{\tau,\delta}$ is supported on $[-\delta/2,\delta/2]$, and 
\[
\widehat{g}_{\tau,\delta}(t) = \frac{\delta}{2} \widehat{g}_{\tau\delta/2}(\delta t/2).
\]

It follows that $t\mapsto (\widehat{g}_{\tau,\delta}(t))^{-1}$ is positive and convex up for $|t|\le \tau$.  Let $a_n = a_n(\tau,\delta)$ be the Fourier coefficients satisfying
\[
\frac{1}{\widehat{g}_{\tau,\delta}(t)} = \sum_{n\in \Z} a_n e^{\pi i \frac{t}{\tau} n}
\]
for $|t|\le \tau$.   Positivity and convexity imply that $|a_n| = (-1)^n a_n$. Define a measure $\nu= \nu_{\tau,\delta}$ on $\R$ for any Borel set $A$ by
\[
\nu(A) = \sum_{n\in\Z} a_n\delta_{n/(2\tau)}(A)
\]
where $\delta_b$ is the Dirac measure at $b\in \R$. We observe that $\widehat{\nu}(t) = 1/\widehat{g}_{\tau,\delta}(t)$ for $|t|\le \tau$, and the total variation satisfies
\[
|\nu|(\R) = \sum_{n\in\Z} |a_n| = \sum_{n\in\Z} a_n(-1)^n = \frac{1}{\widehat{g}_{\tau\delta/2} (\tau\delta/2)}.
\]

It follows that convolution with $\nu$ is the inverse operator of convolution with $g_{\tau,\delta}$ when restricted to $PW_\tau^1$. Moreover, for $g_{\tau,\delta}$ the choice of $\nu$ is optimal, since the value of the Fourier transform of $\nu$ is always a lower bound for the total variation.

 It follows that
\[
\|T_{g_{\tau,\delta}}^{-1}\| = \frac{1}{\widehat{g}_{\tau\delta/2}(\tau\delta/2)}.
\] 

We observe the identities
\[
\frac{\|g_{\tau,\delta}\|_\infty}{\widehat{g}_{\tau,\delta}(\tau)} = \frac{2}{\delta} \frac{\|g_{\tau\delta/2}\|_\infty}{\widehat{g}_{\tau\delta/2}(\tau\delta/2)} =   \frac{2\tau +2\delta^{-1}}{\widehat{g}_{\tau\delta/2}(\tau\delta/2)}.
\]

For $\tau>0$ and $\delta>0$ we use the inequality $\widehat{g}_{\tau\delta/2}(\tau\delta/2) \ge \widehat{g}_0(0) >0.65$. For $\tau\delta \ge 2$, we may use the lower bound $\widehat{g}_1(1) > 0.8$ instead. 
\end{proof}

\section{Proof of Theorem \ref{thm2}}

As in the first section, the main task lies in computing a minimal extrapolation of $1/\widehat{g}$ restricted support of the Fourier transform for a suitably chosen function $g$. For $x\in \R^d$ we consider
\[
g_\alpha(x) = \chi_{B(0,\alpha)}(x)
\]
whose Fourier transform for $t\in\R^d$ is 
\[
\widehat{g}_\alpha(t) = \frac{\alpha^{d/2} J_{d/2}(2\pi \alpha |t|)}{|t|^{d/2}}.
\]

To construct a minimal extrapolation of $1/\widehat{g}_\alpha$ restricted to $[-\tau,\tau]^d$, we need  facts from the theory of Laguerre-P\'olya entire functions. We follow   \cite{HW55}. An entire function $E$ belongs to the Laguerre-P\'olya class $\mc{E}$ if and only if it has the form 
\[
E(s)=e^{cs^2+bs}\prod\limits_{k=1}^{\infty}\left(1-\frac{s}{a_k}\right)e^{\frac{s}{a_k}},$$ where $c\geq 0$, $b$, $a_k(k=1,2,\cdots)$ are real, and $$\sum_{k=1}^{\infty}\frac{1}{a_k^2}<\infty.
\]

\begin{lemma}\label{recip}
There exists a non-negative, integrable function $G$ such that for $x\in (-j_p(1),j_p(1))$
$$\frac{x^p}{J_{p}(x)}=\int\limits_{-\infty}^0 e^{-x^2 t}G(t)dt.$$
\end{lemma}
\begin{proof}
The Bessel function $J_p(x)$ has an infinite product representation  \cite[Section 10.21(iii)]{NIST:DLMF}. Dividing each side by $x^p$ gives us
\begin{align*}
\frac{J_p(x)}{x^p}&=\frac{1}{2^p \Gamma(p+1)}\prod\limits_{k=1}^{\infty}\left(1-\frac{x^2}{j_{p,k}^2}\right)\\
&=\frac{1}{2^p \Gamma(p+1)}\prod\limits_{k=1}^{\infty}\left(1-\frac{x}{j_{p,k}}\right)e^{x/j_{p,k}}\prod\limits_{k=1}^{\infty}\left(1+\frac{x}{j_{p,k}}\right)e^{-x/j_{p,k}}
\end{align*}

Substituting $x=\sqrt{y}$ in the infinite product representation of $\frac{J_p(x)}{x^p}$ gives us 
$$\frac{J_p(\sqrt{y})}{y^{p/2}}=\frac{1}{2^p\Gamma(p+1)}\prod\limits_{k=1}^{\infty}\left(1-\frac{y}{j_{p,k}^2}\right)$$ 
which is an entire function and belongs to class $\mc{E}$. Let $E(x)=\frac{J_p(x)}{x^p}$. Then by \cite[Theorem 6.1]{HW55} the function $1/E(\sqrt{y})$ has a Laplace transform representation given by 
$$\frac{1}{E(\sqrt{y})}=\int\limits_{\R}e^{-y t}G(t)dt,$$
where $G(t)\in C^{\infty}$ is a nonnegative, integrable function and the integral converges in the largest vertical strip which contains the origin and is free of zeroes of $E(\sqrt{y})$, which is $-\infty<y<j_{p,1}^2$. Next, we want to determine the values of $t$ for which $G(t)>0$. This result is obtained using Corollary 3.1 (Chapter 5 in \cite{HW55}). Note that $E(\sqrt{y})$ can be expressed as  
\begin{align*}
E(\sqrt{y})&=\frac{1}{2^p\Gamma(p+1)}\prod\limits_{k=1}^{\infty}\left(1-\frac{y}{j_{p,k}^2}\right)\\
&=\frac{1}{2^p\Gamma(p+1)}\frac{1}{e^{-\sum\limits_{k=1}^{\infty}(y/j_{p,k})}}\prod\limits_{k=1}^{\infty}\left(1-\frac{y}{j_{p,k}^2}\right)e^{(y/j_{p,k})}\\
\end{align*}
The function $E(\sqrt{y})$ has no negative zeroes. Therefore, in the setting of Corollary 3.1, $\alpha_1=-\infty$ and $b=-\sum\limits_{k=1}^{\infty}\frac{1}{j_{p,k}}$. Therefore $G(t)>0$ if $t\in(-\infty,0))$ and $G(t)=0$ otherwise, giving us for $y\in (0,j_{p,1}^2)$,
$$\frac{1}{E(\sqrt{y})}=\int\limits_{-\infty}^0 e^{-yt}G(t)dt.$$
Substituting $y = x^2$ gives the claim.
\end{proof}

Let $\lambda>0$ and $\alpha>0$ with $2\pi\sqrt{d} \alpha\lambda<j_{d/2}(1)$. We construct a (signed) measure $\nu$ that is an inverse transform on $\mc{B}_1([-\lambda,\lambda]^d)$ of convolution with $g_\alpha$ satisfying
\[
\|f*\nu\|_1 \le \frac{(\sqrt{d} \lambda)^{d/2} }{\alpha^{d/2} J_{d/2}(2\pi\sqrt{d} \lambda \alpha)}\|f\|_1,
\]
and we show that the constant is best possible among all inverse transformations of convolution with $g_\alpha$ on $\mc{B}_1([-\lambda,\lambda]^d)$.  We expand $1/\widehat{g}_\alpha$ restricted to $[-\lambda,\lambda]^d$ into its Fourier series
\[
\frac{1}{\widehat{g}_\alpha(t)} = \sum_{n\in\Z^d} H_\alpha(n) e^{2\pi i n t}
\]
where
\[
H_\alpha(n) = \left(\frac{1}{2\lambda}\right)^d \int_{[-\lambda,\lambda]^d} \frac{|x|^{d/2}}{\alpha^{d/2} J_{d/2}(2\pi \alpha|x|)} e^{-i\frac\pi\lambda nx} dx.
\]

\begin{lemma} The coefficients satisfy
\[
H(n_1,...,n_d) = (-1)^{n_1+...+n_d} |H(n_1,...,n_d)|
\]
\end{lemma}

\begin{proof} The restrictions on $\alpha\lambda$ imply that the following integrals converge absolutely. Inserting the Schoenberg representation \eqref{} gives with $n = (n_1,...,n_d)$ 
\[
H_\alpha(n) = \left( \frac{1}{2\sqrt{2\pi} \alpha\lambda}\right)^d \int_{-\infty}^0 G(t) \left( \prod_{j=1}^d \int_{[-\lambda,\lambda]} e^{-x_j^2 t} e^{-i\frac\pi\lambda n_j x_j} dx_j\right) dt
\]

Since $t<0$, the function $x_j\mapsto e^{-t x_j^2}$ is positive, symmetric, and convex up. Hence $e^{-i\frac\pi\lambda n_j x_j}$ may be replaced by $\cos(\frac\pi\lambda n_j x_j)$. A short argument involving two integration by parts may be used to show that 
\[
(-1)^{n_j} \int_{[-\lambda,\lambda]} e^{-x_j^2 t} \cos(\tfrac\pi\lambda n_j x_j ) dx_j \ge 0,
\]
which implies the claim of the lemma.
\end{proof}

We define a measure $\nu_\alpha$ on $\R^d$ by
\[
\nu_\alpha = \sum_{n\in \Z^d} H_\alpha(n) \delta_{\frac{n}{2\lambda}}
\]
where $\delta_x$ is the point measure at $x$ with $\delta_x(\R^d) =1$.

\begin{lemma}\label{d-dim-conv-inv} Let $\lambda$ and $\alpha$ be positive with $2\pi\sqrt{d} \lambda\alpha< j_{d/2}(1)$. Convolution with $\nu_\alpha$ is the inverse operator of convolution with $g_\alpha$ on $\mc{B}_1([-\lambda,\lambda]^d$ with
\[
\|f*\nu_\alpha\|_1 \le \frac{(\sqrt{d} \lambda)^{d/2}}{\alpha^{d/2} J_{d/2}(2\pi\sqrt{d} \alpha\lambda)} \|f\|_1
\]
for all $f\in \mc{B}_1([-\lambda,\lambda]^d)$. 
\end{lemma}

\begin{proof} By construction of $\nu_\alpha$ we have
\[
\widehat{g}_\alpha(t) \widehat{\nu}_\alpha(t) = 1
\]
for all $t\in [-\lambda,\lambda]^d$, and we observe that the total variation measure $|\nu_\alpha|$ satisfies
\[
|\nu_\alpha|(\R^d) = \sum_{n\in\Z^d} |H_\alpha(n)| = \sum_{n\in\Z^d} H_\alpha(n) (-1)^{n_1+...+n_d} = \frac{1}{\widehat{g}_\alpha(\lambda,...,\lambda)},
\]
and Minkowski's inequality $\|f*\nu_\alpha\|_1 \le |\nu_\alpha| \, \|f\|_1$ shows that convolution with $\nu_\alpha$ defines a bounded operator on $\mc{B}_1([-\lambda,\lambda]^d)$ that inverts convolution with $g_\alpha$. 
\end{proof}

 Lemma \ref{d-dim-conv-inv} gives a bound for the operator norm of the inverse of convolution with $g_\alpha$, and the calculation at the beginning of the proof of Theorem \ref{thm1} may be used to complete the proof of Theorem \ref{thm2}.



\bibliographystyle{amsplain}

\bibliography{references}

\end{document}